\documentclass{amsart}
\title{An identification of the Baum--Connes and Davis--L\"uck assembly maps}
\thanks{This work was funded by the Deutsche Forschungsgemeinschaft (DFG, German ResearchFoundation) – Project-ID 427320536 – SFB 1442, as well as under Germany’s Excellence Strategy EXC 2044 390685587, Mathematics Münster: Dynamics–Geometry–Structure.}

\author{Julian Kranz}
\address{Julian Kranz\\
Westf\"alische Wilhelms-Universit\"at M\"unster,
Mathematisches Institut\\
Einsteinstr.~62, 48149 M\"unster, Germany}
\email{julian.kranz@uni-muenster.de}
\urladdr{https://www.uni-muenster.de/IVV5WS/WebHop/user/j\_kran05/}

\usepackage{amssymb,url}
\usepackage{graphicx} 
\usepackage{tikz-cd} 
\usepackage{amsmath}
\usepackage{mathtools} 
\usepackage{bm} 

		\newcommand{\C}{\mathbb C}

		\newcommand{\K}{\bm K}

		\newcommand{\N}{\mathbb N}
		
		\newcommand{\R}{\mathbb R}

		\newcommand{\Z}{\mathbb Z}
		
		\DeclareMathOperator{\Fin}{Fin}
		\DeclareMathOperator{\id}{id}		
		\DeclareMathOperator{\Ob}{Ob}

		\DeclareMathOperator{\op}{op}	
		\DeclareMathOperator{\Or}{Or}	
		\DeclareMathOperator{\pt}{pt}

		\DeclareMathOperator{\Ind}{Ind}
		
		\DeclareMathOperator{\Res}{Res}
		\DeclareMathOperator{\colim}{colim}
		
		\DeclareMathOperator{\Cone}{Cone}
		\DeclareMathOperator{\ev}{ev}

		\DeclareMathOperator{\pr}{pr}
		
		\newcommand{\Sp}{\mathfrak{Sp}}

		\newcommand{\Top}{\mathfrak{Top}}
		
		\newcommand{\grc}{\mathfrak C^*_{\Z_2}}
		
\theoremstyle{plain}
\newtheorem{thm}[subsection]{Theorem}
\newtheorem{lem}[subsection]{Lemma}
\newtheorem{cor}[subsection]{Corollary}
\newtheorem{prop}[subsection]{Proposition}
\theoremstyle{definition}
\newtheorem{defn}[subsection]{Definition}
\newtheorem{rem}[subsection]{Remark}
\newtheorem{eg}[subsection]{Example}

\usepackage{enumitem}
\setenumerate{label=(\roman*),itemsep=3pt,topsep=3pt}

\begin{document}

\begin{abstract}
The Baum--Connes conjecture predicts that a certain assembly map is an isomorphism. We identify the homotopy theoretical construction of the assembly map by Davis and L\"uck \cite{davis1998spaces} with the category theoretical construction by Meyer and Nest \cite{meyer2004baum}. This extends the result of Hambleton and Pedersen \cite{hambleton2004identifying} to arbitrary coefficients. Our approach uses abstract properties rather than explicit constructions and is formally similar to Meyer's and Nest's identification of their assembly map with the original construction of the assembly map by Baum, Connes and Higson \cite{baum1994classifying}.
\end{abstract}

\maketitle

\section{Introduction}

Let $G$ be a countable discrete group and $A$ a separable $G$-$C^*$-algebra. The Baum--Connes conjecture predicts that the \emph{Baum--Connes assembly map} 
	\[\mu: K^G_*(\mathcal E_{\Fin}G,A)\to K_*(A\rtimes_r G)\]
is an isomorphism. The map was defined by Baum, Connes and Higson \cite{baum1994classifying} using the equivariant $KK$-theory of Kasparov \cite{Kasparov1988}. Later, a homotopy theoretical definition of the assembly map was given by Davis and L\"uck \cite{davis1998spaces}. They developed an abstract machinery to study isomorphism conjectures like the Baum--Connes conjecture or the Farrell-Jones conjecture in a common framework. Their machinery takes as input a family $\mathcal F$ of subgroups of $G$ and an $\Or(G)$-spectrum $\bm E$, i.e. a functor from the category of all homogeneous $G$-spaces $G/H$ to the category of spectra. Every $\Or(G)$-spectrum $\bm E$ has a natural extension to the category of $G$-$CW$-complexes and defines a $G$-equivariant homology theory $H^G_*(-,\bm E)$ by taking homotopy groups. In this setting, the \emph{$(\bm E,\mathcal F,G)$-assembly map} is the map 
	\begin{equation}\label{generalassembly}
		H^G_*(\mathcal E_\mathcal FG,\bm E)\to H^G_*(\pt,\bm E)
	\end{equation}
induced by the projection $\mathcal E_\mathcal FG\to \pt$ where $\mathcal E_\mathcal FG$ denotes a classifying space for the family $\mathcal F$. \\

To obtain the Baum--Connes assembly map in \eqref{generalassembly}, one takes $\mathcal F=\Fin$ to be the family of finite subgroups and $\K_A^G$ to be an $\Or(G)$-spectrum satisfying
	\begin{equation}\label{KtheorySpectrumOnG/H}
		\pi_*(\K_A^G(G/H))\cong K_*(A\rtimes_r H)
	\end{equation}
for all subgroups $H\subseteq G$. We call the resulting assembly map
	\[H^G_*(\mathcal E_{\Fin}G,\K_A^G)\to H^G_*(\pt,\K_A^G)\]
the \emph{Davis--L\"uck assembly map}. The construction of $\K_A^G$ has been done by Davis and L\"uck in the case $A=\C$ and by Mitchener \cite{mitchener2002c} in the general case. We will give a variant of Mitchener's construction using Michael Joachims $K$-theory spectrum for $C^*$-categories \cite{joachim2003k}.\\

It is not at all obvious that this construction gives rise to the same assembly map as in \cite{baum1994classifying}. Identifications have been made in \cite{hambleton2004identifying} for the case $A=\C$ and in \cite{mitchener2002c} for the general case. However, both works rely on heavy machinery and omit a lot of detail. Furthermore, the construction of the assembly map in \cite{mitchener2002c} contains some inconsistencies. For example, it is not clear to the author of this paper whether the $K$-theory class $[\mathcal E_K]$ in \cite[Def. 6.2]{mitchener2002c} is well-defined for a noncompact $G$-space $K$. \\

The main ingredient for our identification is yet another construction of the assembly map by Meyer and Nest \cite{meyer2004baum}. Recall that the equivariant $KK$-groups $KK^G(A,B)$ are the morphism sets of a triangulated category $\mathfrak{KK}^G$ with separable $G$-$C^*$-algebras as objects. Let $\mathcal{CI}\subseteq \mathfrak{KK}^G$ be the full subcategory of $G$-$C^*$-algebras $\Ind_H^GB$ induced from finite subgroups $H\subseteq G$. Let $\langle \mathcal {CI}\rangle$ be the localizing subcategory generated by $\mathcal {CI}$, i.e. the smallest full subcategory containing $\mathcal {CI}$ which is closed under $KK^G$-equivalence, suspension, mapping cones and countable direct sums. Every $G$-$C^*$-algebra can be approximated by a $G$-$C^*$-algebra in $\langle \mathcal CI\rangle$ in the following sense:

	\begin{thm}[{\cite[Prop. 4.6]{meyer2004baum}}]
		Let $A$ be a separable $G$-$C^*$-algebra. Then there is a $G$-$C^*$-algebra $\tilde A\in \langle \mathcal{CI}\rangle$ and an element $D\in KK^G(\tilde A,A)$ which restricts to a $KK^H$-equivalence for every finite subgroup $H\subseteq G$.
	\end{thm}
	
Meyer and Nest identify the Baum--Connes assembly map with the map
	\begin{equation*}\label{triangulatedBC}
		D_*:K_*(\tilde A\rtimes_rG)\to K_*(A\rtimes_rG),
	\end{equation*}
which we call the \emph{Meyer--Nest assembly map}. In fact, they achieve the identification as follows:

\begin{thm}[{\cite[Thm. 5.2]{meyer2004baum}}]\label{meyernestidentification}
	The indicated maps in the following diagram are isomorphisms.
	\begin{equation*}
		\begin{tikzcd}
			&K^G_*(\mathcal E_{\Fin}G,\tilde A)\arrow[r,"\cong","\mu"']\arrow[d,"\cong","D_*"'] &K_*(\tilde A\rtimes_rG)\arrow[d,"D_*"]\\
			&K^G_*(\mathcal E_{\Fin}G,A)\arrow[r,"\mu"]					&K_*(A\rtimes_rG)
		\end{tikzcd}.
	\end{equation*}
\end{thm}

We use the same strategy, to identify the Davis--L\"uck assembly map to the Meyer--Nest assembly map:

\begin{thm}[Theorem \ref{mainidentification}] \label{intromaintheorem}
	The indicated maps in the following diagram are isomorphisms:
		\begin{equation}\label{DL-BC-diagram}
		\begin{tikzcd}
			&H^G_*(\mathcal E_{\Fin}G,\K_{\tilde A}^G)\arrow[r,"\cong","\pr_*"']\arrow[d,"\cong","D_*"']	&H^G_*(\pt,\K_{\tilde A}^G)\arrow[d,"D_*"]\\
			&H^G_*(\mathcal E_{\Fin}G,\K_A^G)\arrow[r,"\pr_*"]				&H^G_*(\pt,\K_A^G)
		\end{tikzcd}.
	\end{equation}
\end{thm}

Here the lower-hand map is the Davis--L\"uck assembly map and the right hand map is identical to the Meyer--Nest assembly map by \eqref{KtheorySpectrumOnG/H}. 

Let us outline the proof of the above theorem. First we prove that the map
	\[H^G_*(G/H,\K_{\tilde A}^G)\to H^G_*(G/H,\K_A^G)\]
is an isomorphism for any finite subgroup $H\subseteq G$. Indeed, by \eqref{KtheorySpectrumOnG/H} this map can be identified with the map
	\[K_*(\tilde A\rtimes_r H)\to K_*(A\rtimes_r H).\]
It is an isomorphism since $D\in KK^G(\tilde A,A)$ is a $KK^H$-equivalence. Using excision we conclude that the map 
	\[H^G_*(\mathcal E_{\Fin}G,\K_{\tilde A}^G)\to H^G_*(\mathcal E_{\Fin}G,\K_A^G)\]
is an isomorphism as well.

To prove that the upper-hand map in \eqref{DL-BC-diagram} is an isomorphism, we proceed in two steps: First we show that the class of all $\tilde A\in \mathfrak {KK}^G$, for which it \emph{is} an isomorphism, is localizing. This boils down to translating $KK^G$-equivalences, suspensions, mapping cone sequences and direct sums in $\mathfrak {KK}^G$ to stable equivalences, loops, fiber sequences and wedge sums in spectra. The next step is to show that the upper-hand map in \eqref{DL-BC-diagram} is an isomorphism for all generators $\tilde A=\Ind_H^GB\in \mathcal{CI}$. To see this, we use Green's imprimitivity theorem to construct a natural induction isomorphism
	\begin{equation}\label{introinduction}
		 H^H_*(X|_H,\K_B^H)\cong H^G_*(X,\K_{\Ind_H^GB}^G),
	\end{equation}
for any $G$-$CW$-complex $X$. We can then identify the map in question with the map
	\[H^H_*(\mathcal E_{\Fin}G|_H,\K_B^H)\to H^H_*(\pt,\K_B^H).\]
This map is an isomorphism since $H$ is finite.\\

While this work was published, our main result was proved independently by Bunke, Engel and Land \cite{bunke2021paschke} with completely different methods. 
 
\subsection*{Outline of the paper}
The paper is organized as follows: In section \ref{meyer-nest} we describe the category $\mathfrak{KK}^G$ and recall the construction of the Meyer--Nest assembly map. Section \ref{davis-lueck} contains the construction of equivariant homology theories and assembly maps from $\Or(G)$-spectra as well as some basic homotopy theory for $\Or(G)$-spectra. The results are well-known and can be found either explicitly or implicitly in \cite{davis1998spaces} and \cite{luck2003commuting}. But we hope that including them keeps the exposition reasonably self-contained. In section \ref{kspectrum} we construct the $\Or(G)$-spectrum $\K_A^G$. We begin by discussing groupoid $C^*$-algebras and their reduced crossed products. For better functoriality properties, we consider the reduced crossed product of a groupoid $C^*$-algebra as a $C^*$-\emph{category} rather than
a $C^*$-algebra. Our construction is similar to the construction in \cite{mitchener2002c}. We then recall the construction of Michael Joachims $K$-theory spectrum $\K$ for $C^*$-categories (see \cite{joachim2003k}). Finally we define $\K_A^G$ by the formula
	\[\K_A^G(G/H):=\K(A\rtimes_r \overline{G/H}),\]
where $\overline {G/H}$ denotes the transformation groupoid associated to the $G$-space $G/H$. We end the section by discussing some homotopy theoretical properties of the functor $A\mapsto \K_A^G$. In section \ref{identification} we use all the technology developed so far to construct the induction isomorphism \eqref{introinduction} and to prove Theorem \ref{intromaintheorem}. We include a discussion on variants of our results for other crossed product functors in section \ref{exotic}. \\

\subsection*{Notation}
If $\mathcal C$ is a category, we denote its homomorphism sets by $\mathcal C(x,y)$, its collection of objects by $\Ob(\mathcal C)$ and its opposite category by $\mathcal C^{\op}$. All $C^*$-algebras are complex. If $A$ is a $C^*$-algebra, we denote by $M(A)$ its multiplier algebra and by $Z(A)$ its center. If $X$ is a set, we denote by $\ell^2(X,A)$ the right Hilbert-$A$-module $\oplus_{x\in X}A$ and by $\mathcal L_A(\ell^2(X,A))$ its adjointable operators.

	\section{Meyer--Nest theory}\label{meyer-nest}
In this section, we recall the basic properties of equivariant $KK$-theory and the definition of the Meyer--Nest assembly map. Throughout this section, $G$ is a countable discrete group and all $C^*$-algebras are assumed to be separable. By a \emph{$G$-Hilbert space} we mean a Hilbert space $\mathcal H$ together with a unitary representation $u:G\to U(\mathcal H)$. We denote the algebra of compact operators on $\mathcal H$ by $\mathcal K(\mathcal H)$ and equip it with the $G$-action given by conjugation with $u$. We denote by $\mathfrak C^*_G$ the category of all separable $G$-$C^*$-algebras with $G$-equivariant $*$-homomorphisms. For two $G$-$C^*$-algebras $A$ and $B$, the tensor product $A\otimes B$ denotes the \emph{minimal} tensor product with the natural $G$-action. We denote the reduced crossed product of $A$ and $G$ by $A\rtimes_r G$. The \emph{suspension} of $A$ is the $G$-$C^*$-algebra $SA:= C_0((0,1))\otimes A\cong C_0((0,1),A)$ with the trivial $G$-action on the first factor. The \emph{mapping cone} of a morphism $\pi:A\to B$ is given by
	\[\Cone(\pi):= \{(a,b)\in A\oplus C_0((0,1],B)\mid \pi(a)=b(1)\}.\]
The \emph{mapping cone triangle} associated to $\pi$ is the sequence
	\[SB\to \Cone(\pi)\to A\xrightarrow{\pi} B\]
where the first map is given by inclusion and the second map is given by evaluation at $1$. We also call $\Cone(\pi)\to A\to B$ a \emph{mapping cone sequence}. A short exact sequence 
	\[0\to I\to A\xrightarrow{\pi} B\to 0\]
of $G$-$C^*$-algebras is called \emph{split exact}, if there is a $G$-equivariant $*$-homomorphism $\sigma:B\to A$ satisfying $\pi \sigma =\id_B$. For a subgroup $H\subseteq G$, we denote by $\Res_G^H:\mathfrak C^*_G\to \mathfrak C^*_H$ the obvious restriction functor. Let $B$ be an $H$-$C^*$-algebra with $H$-action $\beta$. The \emph{induced algebra} $\Ind_H^GB$ is the $C^*$-algebra of all bounded functions $f:G\to B$ satisfying $f(gh)=\beta_{h^{-1}}(f(g))$ for all $g\in G$ and $h\in H$, such that the function $gH\mapsto \|f(gH)\|$ belongs to $C_0(G/H)$. We equip $\Ind_H^GB$ with the $G$-action given by left translation.

The following theorem is a collection of well-known results on equivariant $KK$-theory. For more details we refer to \cite{meyer2008categorical} and the references therein.
\begin{thm}
	There is an additive category $\mathfrak {KK}^G$ with the same objects as $\mathfrak C^*_G$ and a functor 
		$KK^G:\mathfrak C^*_G\to \mathfrak{KK}^G$
	with the following properties:
		\begin{enumerate}
			\item $KK^G$ is $G$-homotopy invariant.
			\item For any two separable $G$-Hilbert spaces $\mathcal H,\mathcal H'$ and any $G$-$C^*$-algebra $A$, the stabilization morphism
				\[A\otimes \mathcal K(\mathcal H)\to A\otimes \mathcal K(\mathcal H\oplus \mathcal H')\]
			is mapped to an isomorphism in $\mathfrak {KK}^G$.
			\item Any split exact sequence $0\to I\to A\to B\to 0$ of $G$-$C^*$-algebras is mapped to a split exact sequence in $\mathfrak {KK}^G$.
			\item $KK^G:\mathfrak C^*_G\to \mathfrak {KK}^G$ is universal with the above properties in the sense that any other functor from $\mathfrak C^*_G$ into an additive category with the above properties uniquely factors through $KK^G$.
			\item The category $\mathfrak {KK}^G$ is triangulated with respect to the suspension functor $S$ and the mapping cone triangles 
				\[SB\to \Cone(\pi)\to A\xrightarrow{\pi} B.\]
				We write $KK^G_n(A,B):= KK^G(A,S^nB):=\mathfrak{KK}^G(A,S^n B)$. 
			\item We have Bott periodicity: $KK^G_n(A,B)\cong KK^G_{n+2}(A,B)$.
			\item Topological $K$-theory is given by $K_*(A)\cong KK_*(\C,A):=KK_*^{\{e\}}(\C,A)$. 
			\item Let $H\subseteq G$ be a subgroup and $A$ a $G$-$C^*$-algebra. Then the functors 
				\begin{equation}\label{functorsinC*}
				\begin{aligned}
					\Ind_H^G:&~\mathfrak C^*_H \to \mathfrak C^*_G \\
					\Res_G^H:&~\mathfrak C^*_G \to \mathfrak C^*_H \\
					\rtimes_r G: &~\mathfrak C^*_G\to \mathfrak C^*\\
					\otimes A: &~\mathfrak C^*_G \to \mathfrak C^*_G
				\end{aligned}
				\end{equation}
				
				uniquely extend to functors 
				\begin{equation}\label{functorsinKK}
				\begin{aligned}
					\Ind_H^G:&~\mathfrak{KK}^H \to \mathfrak{KK}^G \\
					\Res_G^H:&~\mathfrak{KK}^G \to \mathfrak{KK}^H \\
					\rtimes_r G: &~\mathfrak{KK}^G\to \mathfrak{KK}\\
					\otimes A: &~\mathfrak{KK}^G \to \mathfrak{KK}^G.
				\end{aligned}
				\end{equation}
				Natural transformations between the functors in \eqref{functorsinC*} are in bijection with natural transformations between the corresponding functors in \eqref{functorsinKK}. Furthermore, $\Ind_H^G:\mathfrak {KK}^H\to \mathfrak {KK}^G$ is left adjoint to $\Res_G^H:\mathfrak {KK}^G\to \mathfrak {KK}^H$. 
		\end{enumerate}
\end{thm}

The isomorphisms in $\mathfrak {KK}^G$ are also called \emph{$KK^G$-equivalences}. A $G$-$C^*$-algebra $A$ is called \emph{$KK^G$-contractible}, if it is isomorphic to $0$ in $\mathfrak{KK}^G$.

\begin{defn}\label{localizingdefinition}
	A full subcategory $\mathcal C\subseteq \mathfrak {KK}^G$ is called \emph{localizing}, if it is closed under $KK^G$-equivalence, suspension, mapping cones and countable direct sums. Being closed under mapping cones means that if $\Cone(\pi)\to A\xrightarrow{\pi} B$ is a mapping cone sequence and if $A$ and $B$ belong to $\mathcal C$, then $\Cone(\pi)$ also belongs to $\mathcal C$.
\end{defn}

Since exact triangles may be rotated, the algebras $\Cone(\pi),A$ and $B$ belong to a localizing subcategory $\mathcal C$ if at least two of them belong to $\mathcal C$. The restriction to \emph{countable} direct sums in the above definition is necessary in order to stay in the realm of separable $C^*$-algebras. For any full subcategory $\mathcal C\subseteq \mathfrak {KK}^G$, there is a smallest localizing subcategory $\langle \mathcal C\rangle \subseteq \mathfrak{KK}^G$ containing $\mathcal C$.

\begin{defn}[{\cite[Def. 4.1]{meyer2004baum}}]
	Let $\mathcal {CI}\subseteq \mathfrak {KK}^G$ denote the full subcategory of $G$-$C^*$-algebras of the form $\Ind_H^GB$, where $H\subseteq G$ is a finite subgroup and $B$ is an $H$-$C^*$-algebra. 
	Let $\mathcal {CC}\subseteq \mathfrak {KK}^G$ denote the full subcategory of $G$-$C^*$-algebras $N$, such that $N$ is $KK^H$-contractible for any finite subgroup $H\subseteq G$. 
\end{defn}

\begin{thm}[{\cite[Thm. 4.7]{meyer2004baum}}]\label{Meyernesttheorem}
	The localizing subcategories $\langle \mathcal {CI}\rangle\subseteq \mathfrak{KK}^G$ and $ \mathcal{CC} \subseteq \mathfrak {KK}^G$ are complementary in the following sense:
		\begin{enumerate}
			\item For any $A\in \langle \mathcal{CI}\rangle$ and $B\in \mathcal{CC}$, we have $KK^G(A,B)=0$.
			\item For any $G$-$C^*$-algebra $A$, there is an exact triangle 
		\[SN\to \tilde A\xrightarrow{D} A\to N\]
	with $N\in \mathcal {CC}$ and $\tilde A\in \langle \mathcal {CI}\rangle$. The above triangle is unique up to isomorphism.
		\end{enumerate}		 
\end{thm}

\begin{rem}
	The morphism $D:\tilde A\to A$ is called the \emph{Dirac morphism}. Note that it follows from the adjunction of $\Ind_H^G$ and $\Res_G^H$ that $D$ is a $KK^H$-equivalence for any finite subgroup $H\subseteq G$. 
\end{rem}

\begin{thm}[{\cite[Thm. 5.2]{meyer2004baum}}]
	The indicated maps in the following diagram are isomorphisms.
	\begin{equation*}
		\begin{tikzcd}
			&K^G_*(\mathcal E_{\Fin}G,\tilde A)\arrow[r,"\cong","\mu"']\arrow[d,"\cong","D_*"'] &K_*(\tilde A\rtimes_rG)\arrow[d,"D_*"]\\
			&K^G_*(\mathcal E_{\Fin}G,A)\arrow[r,"\mu"]					&K_*(A\rtimes_rG)
		\end{tikzcd}
	\end{equation*}
	In particular, the Baum--Connes assembly map can canonically be identified with the map
		\begin{equation*}
			D_*:K_*(\tilde A\rtimes_r G)\to K_*(A\rtimes_r G).
		\end{equation*}
\end{thm}
	
We call the above map the \emph{Meyer--Nest assembly map.}

\section{Davis--L\"uck theory}\label{davis-lueck}

In this section, we recall the basic machinery of \cite{davis1998spaces} in order to write down the Davis--L\"uck assembly map. We also state some homotopy theoretical results which will allow us to prove that the class of $G$-$C^*$-algebras, for which the Davis--L\"uck assembly map is an isomorphism, is localizing. \\

Throughout this section, we work in the category of compactly generated weak Hausdorff spaces with continuous maps (see \cite{strickland2009compactly}) and denote this category by $\Top$. Similarly, we denote the category of pointed compactly generated weak Hausdorff spaces with pointed continuous maps by $\Top_*$. These categories are closed symmetric monoidal with respect to the product $X\times Y$ respectively the smash product $X\wedge Y$. We denote the mapping spaces by $\Top(X,Y)$ respectively $\Top_*(X,Y)$. We write $X_+:= X\coprod \{+\}$ to equip a space $X$ with a disjoint basepoint $+$ and reserve the notation $Y^+$ for the one-point compactification of a locally compact space $Y$. We use the notation $\Omega X:=\Top_*(S^1,X)$ and $\Sigma X:= S^1\wedge X$ to denote the loop space and the suspension of a pointed space $X$. Recall that there is a natural adjunction homeomorphism
	\begin{equation}\label{mappingloopsadjunction}
		\Top_*(\Sigma X,Y)\cong \Top_*(X,\Omega Y).
	\end{equation}
We denote by $\pi_n(X):=\pi_0(\Omega^n X), n\geq 0$ the \emph{$n$-th} homotopy group of a pointed space $X$. A pointed map is called a \emph{weak equivalence}, if it induces an isomorphism on all homotopy groups. For a discrete group $G$, we denote by $\Top^G$ the category of (compactly generated weak Hausdorff) $G$-spaces and $G$-equivariant maps. We equip the mapping spaces $\Top^G(X,Y)$ with the topology inherited from the inclusion $\Top^G(X,Y)\subseteq \Top(X,Y)$. 

\subsection*{Spaces and spectra over the orbit category}

\begin{defn}
	A spectrum $\bm E$ is a sequence of pointed spaces $\bm E_n,n\geq 0$ together with pointed maps $\bm E_n\to \Omega \bm E_{n+1}$ called \emph{structure maps}. A map $f:\bm E\to \bm F$ of spectra is a sequence of pointed maps $f_n:\bm E_n\to \bm F_n$ which commute with the structure maps. We denote the category of spectra by $\Sp$. 
\end{defn}

\begin{defn}
	Let $\bm E$ be a spectrum. For $n\in \Z$, the \emph{$n$-th homotopy group} of $\bm E$ is the group
		\[\pi_n(\bm E):=\underset{k\to \infty}\colim\pi_{n+k}(\bm E_k).\]
	Here the colimit is taken with respect to the maps 
		\[\pi_{n+k}(\bm E_k)\to \pi_{n+k}(\Omega \bm E_{k+1})\cong \pi_{n+k+1}(\bm E_{k+1}).\]
	A map of spectra is called a \emph{stable equivalence}, if it induces an isomorphism on all homotopy groups.
\end{defn}

\begin{defn}
	Let $G$ be a discrete group. The \emph{orbit category} $\Or(G)$ is the category of all homogeneous $G$-sets $G/H$ together with $G$-equivariant maps. 
\end{defn}

\begin{defn}[{\cite[Def. 1.2]{davis1998spaces}}]
	A pointed $\Or(G)$-space is a functor $X:\Or(G)\to \Top_*$. A map of pointed $\Or(G)$-spaces is a natural transformation of the underlying functors. Analogously, we define (pointed) $\Or(G)^{\op}$-spaces and $\Or(G)$-spectra. 
\end{defn}

\begin{eg}
	Let $X$ be a $G$-space. We can define a pointed $\Or(G)^{\op}$-space 
		\[G/H\mapsto \Top^G(G/H,X)_+\cong X^H_+\]
	where $X^H\subseteq X$ denotes the space of $H$-fixed-points.
\end{eg}

\begin{defn}[{\cite[Def. 1.4]{davis1998spaces}}]\label{balancedsmashproduct}
	Let $X$ be a pointed $\Or(G)^{\op}$-space and $Y$ a pointed $\Or(G)$-space. The \emph{balanced smash product} of $X$ and $Y$ is the pointed space
		\[X\wedge_{\Or(G)}Y:=\left ( \bigvee_{G/H\in \Or(G)}X(G/H)\wedge Y(G/H)\right)/\sim\]
	where the equivalence relation $\sim$ is generated by the relations
		\[f^*x\wedge y\sim x\wedge f_*y,\quad x\in X(G/H), y\in Y(G/K), f\in \Or(G)(G/K,G/H).\]
	If $\bm E$ is an $\Or(G)$-spectrum, we define the \emph{balanced smash product} ${X\wedge_{\Or(G)}\bm E}$ of $X$ and $\bm E$ as the spectrum given by the sequence of pointed spaces ${X\wedge_{\Or(G)}\bm E_n}, {n\in \N}$ with structure maps given by the adjoints of the natural maps 
		\[(X\wedge_{\Or(G)} \bm E_n)\wedge S^1\cong X\wedge_{\Or(G)}(\bm E_n\wedge S^1)\to X\wedge \bm E_{n+1}\]
		under the adjunction \eqref{mappingloopsadjunction}.
\end{defn}

\begin{defn}[{cp. \cite[Def. 4.3]{davis1998spaces}}]
	Let $X$ be a $G$-$CW$-complex and $\bm E$ an $\Or(G)$-spectrum. The \emph{$G$-equivariant homology of $X$ with coefficients in $\bm E$} is given by
		\[H^G_*(X,\bm E):=\pi_*(\Top^G(-,X)_+\wedge_{\Or(G)}\bm E).\]
\end{defn}

\begin{rem}\label{orGextension}
	Note that there is a natural homeomorphism 
		\[\Top^G(-,G/H)_+\wedge_{\Or(G)}\bm E\to \bm E(G/H),\quad f\wedge x\mapsto f_*(x)\]
	for any subgroup $H\subseteq G$. In particular, we have a natural isomorphism 
		\[H^G_*(G/H,\bm E)\cong \pi_*\bm E(G/H).\]
\end{rem}

\begin{prop}[{\cite[Lem. 4.4]{davis1998spaces}}]
	The functor $H^G_*(-,\bm E)$ defines a generalized homology theory for $G$-$CW$-complexes. 
\end{prop}

\begin{defn}
	A collection $\mathcal F$ of subgroups of $G$ is called a \emph{family of subgroups}, if it is closed under conjugation and taking subgroups. A \emph{classifying space for $\mathcal F$} is a $G$-$CW$-complex $\mathcal E_\mathcal FG$ such that the fixed points $(\mathcal E_\mathcal FG)^H$ with respect to a subgroup $H\subseteq G$ are contractible for $H\in \mathcal F$ and empty for $H\notin \mathcal F$. 
\end{defn}

\begin{lem}[{\cite[Sec. 7]{davis1998spaces}}]
	For any family $\mathcal F$ of subgroups of $G$, there is a classifying space $\mathcal E_\mathcal F G$. Furthermore, $\mathcal E_\mathcal F G$ is unique up to $G$-homotopy equivalence. 
\end{lem}

\begin{defn}[{\cite[Sec. 5.1]{davis1998spaces}}]
	Let $G$ be a discrete group, $\mathcal F$ a family of subgroups and $\bm E$ an $\Or(G)$-spectrum. The \emph{$(\bm E,\mathcal F,G)$-assembly map} is the map
		\[H^G_*(\mathcal E_\mathcal F G,\bm E)\to H^G_*(\pt, \bm E)\]
	induced by the projection $\mathcal E_\mathcal F G\to \pt$. 
\end{defn}


The following lemma is a special case of \cite[Lem. 1.9]{davis1998spaces}.
\begin{lem}\label{davisinduction}
	Let $H\subseteq G$ be a subgroup. Consider the induction functor 
		\[I:\Or(H)\to \Or(G),\quad H/K\mapsto G\times_HH/K\cong G/K.\]
	Let $\bm E$ be an $\Or(H)$-spectrum and denote by $I_*\bm E$ the $\Or(G)$-spectrum given by
		\[I_*\bm E(G/K):= \Top^G(I(-),G/K)_+\wedge_{\Or(H)}\bm E\cong \Top^H(-,G/K)_+\wedge_{\Or(H)}\bm E.\]	
	 Let $X$ be a $G$-$CW$-complex and denote by $X|_H$ the same space with the action restricted to $H$. Then there is a  natural isomorphism
		\[H^H_*(X|_H,\bm E)\cong H^G_*(X,I_*\bm E).\]
\end{lem}

\subsection*{Homotopy theory for $\Or(G)$-spectra}

The last part of this section deals with those homotopy theoretical statements which guarantee that the class of $G$-$C^*$-algebras $A$, for which the $(\K_A^G,\Fin,G)$-assembly map is an isomorphism, is localizing. We recall some basic homotopy theoretical terminology and refer to \cite{tom2008algebraic} for more details.

Let $f:X\to Y$ be a pointed map and let $x_0\in X,y_0\in Y$ be the basepoints. We denote by $C_f$ the cone of $f$ that is the pointed space obtained from $(X\times [0,1])\cup Y$ by gluing $X\times \{1\}$ to $Y$ along $f$ and by collapsing $X\times \{0\}\cup \{x_0\}\times [0,1]$ to a point. The homotopy fiber of $f$ is the pointed space 
	\[F_f:=\{(x,\gamma)\in X\times \Top([0,1],Y)\mid f(x)=\gamma(1),\gamma(0)=y_0\}\]
whose basepoint is given by $(x_0,y_0)$. Both the cone and the homotopy fiber define functors on a category with pointed maps as objects and commutative squares as morphisms. There are natural pointed homeomorphisms 
	\begin{equation}\label{coneisafunctor}
		C_{\Sigma f}\cong \Sigma C_f,\quad F_{\Omega f}\cong \Omega F_f.
	\end{equation}
We define the cone $C_f$ of a map $f:\bm E\to \bm F$ of spectra as the sequence of spaces $C_{f_n}$ with structure maps given by the adjoints of the maps 
	\[\Sigma C_{f_n}\cong C_{\Sigma f_n}\to C_{f_{n+1}}.\]
The homotopy fiber of $f$ is defined analogously, using the second homeomorphism of \eqref{coneisafunctor}. Let $X\xrightarrow{f} Y \xrightarrow{g}Z$ be a sequence of pointed maps together with a homotopy $h:[0,1]\times X\to Z$ of pointed maps such that $h_1$is equal to $gf$ and $h_0$ is the constant map. We call $X\xrightarrow{f} Y \xrightarrow{g}Z$ a \emph{cofiber sequence}, if the canonical map
	\[C_f\to Z,\quad \begin{cases}(x,t)\mapsto h_t(x)\\y\mapsto g(y)\end{cases}\]
is a weak equivalence. Dually, we call $X\xrightarrow{f} Y \xrightarrow{g}Z$ a fiber sequence, if the canonical map
	\[X\to F_g,\quad x\mapsto (f(x),t\mapsto h_t(x))\]
is a weak equivalence. Note that the homotopy is part of the datum of a (co-)fiber sequence. However, we drop the homotopy from our notation whenever it is clear from context. By replacing (homotopies of) pointed maps by (homotopies of) maps of ($\Or(G)$-)spectra and by replacing weak equivalences by stable equivalences, we obtain analog notions of (co-)fiber sequences of ($\Or(G)$-)spectra.

\begin{lem}[{\cite[Lemma 2.6]{luck2003commuting}}]\label{fibercofiber}
	A sequence $\bm E\to \bm F\to \bm G$ of maps of spectra is a fiber sequence if and only if it is a cofiber sequence. In this case there is a natural long exact sequence 
		\[\dotsb\to\pi_{n+1}(\bm G)\to \pi_n (\bm E)\to \pi_n(\bm F)\to \pi_n(\bm G)\to \pi_{n-1}(\bm F)\to \dotsb\]
	of homotopy groups. 
\end{lem}

The following well-known lemma is an easy consequence of Lemma \ref{fibercofiber}. 
\begin{lem}[{cp. \cite[Prop. 6.12(i)]{schwede2012symmetric}}]\label{homotopydirectsums}
	Let $\bm E_i,i\in I$ be a collection of spectra. Then the natural map 
		\[\pi_*\left(\bigvee_{i\in I}\bm E_i\right)\to \bigoplus_{i\in I}\pi_*(\bm E_i)\]
	is an isomorphism. 
\end{lem}

\begin{lem}[{\cite[Lem. 4.6]{davis1998spaces}}]\label{homotopyinvariance}
	Let $\bm E\to \bm F$ be a stable equivalence of $\Or(G)$-spectra and $X$ a $G$-$CW$-complex. Then the induced map 
		\[H^G_*(X,\bm E)\to H^G_*(X,\bm F)\]
	is an isomorphism.
\end{lem}

The following lemma is inspired by \cite[Def. 3.13]{davis1998spaces}.

\begin{lem}\label{cofibertocofiber}
	Let $X$ be a $G$-$CW$-complex. Then the functor 
		\[\Top^G(-,X)_+\wedge_{\Or(G)}-\]
	maps cofiber sequences of $\Or(G)$-spectra to cofiber sequences of spectra.
	\begin{proof}
		The preceding lemma shows that the functor $\Top^G(-,X)_+\wedge_{\Or(G)}-$ commutes with stable equivalences. It therefore suffices to show that it also commutes with taking cones. To see that this is indeed the case, we reformulate the definition of the cone. Consider the category $\mathcal C$ represented by the following diagram: 
		
		\begin{equation}\label{pushoutcategory}
				\begin{tikzcd}
				c_0\arrow[d]\arrow[r]&c_1\\
				c_2&\\
				\end{tikzcd}
		\end{equation}
		
		There is a natural $\mathcal C^{\op}$-space $E\mathcal C$ given by
			\[E\mathcal C(c_0)=[0,1],\quad E\mathcal C(c_2)=\{0\},\quad E\mathcal C(c_1)=\{1\}\]
		on objects and by the obvious inclusions on morphisms. A morphism
			$f:\bm E\to \bm F$
		of spectra gives rise to a $\mathcal C$-spectrum $\bm D_f$ by mapping diagram \eqref{pushoutcategory} to the following diagram:
		
		\[\begin{tikzcd}
				\bm E\arrow[d]\arrow[r,"f"]&\bm F\\
				\pt&\\
			\end{tikzcd}.\]
			
		Similarly, a morphism $f:\bm E\to \bm F$ of $\Or(G)$-spectra gives rise to a $\mathcal C\times \Or(G)$-spectrum $\bm D_f$. Now the cone of $f$ can be rewritten as 
			\[C_f=E\mathcal C_+\wedge_{\mathcal C}\bm D_f.\]
		Using associativity of balanced smash products and observing that the construction $f\mapsto \bm D_f$ commutes with our functor $\Top^G(-,X)_+\wedge_{\Or(G)}-$, we obtain the desired formula
			\begin{align*}
				  &C_{\Top^G(-,X)_+\wedge_{\Or(G)}f} \\
				= &E\mathcal C_+\wedge_\mathcal C \bm D_{\Top^G(-,X)_+\wedge_{\Or(G)}f} \\
				= &E\mathcal C_+\wedge_\mathcal C (\Top^G(-,X)_+\wedge_{\Or(G)}\bm D_f) \\
				= &\Top^G(-,X)_+\wedge_{\Or(G)}(E\mathcal C_+\wedge_\mathcal C \bm D_f) \\
				= &\Top^G(-,X)_+\wedge_{\Or(G)}C_f
			\end{align*}
		
	\end{proof}
\end{lem}

\section{The $\Or(G)$-spectrum $\textit{\textbf{K}}_A^G$}\label{kspectrum}

In this section, we associate an $\Or(G)$-spectrum $\K_A^G$ to every $G$-$C^*$-algebra $A$, closely following \cite{mitchener2002c} and \cite{joachim2003k}. We call the resulting assembly map 
	\[H^G_*(\mathcal E_{\Fin}G,\K_A^G)\to H^G_*(\pt,\K_A^G)\]
the \emph{Davis--L\"uck assembly map}, where $\Fin$ denotes the family of finite subgroups of $G$. Let us motivate the construction of $\K_A^G$. In order for the right hand sides of the Baum--Connes and Davis--L\"uck assembly maps to match, we need an isomorphism
	\[H^G_*(\pt,\K_A^G)= \pi_*(\K_A^G(G/G))\overset{!}{\cong} K_*(A\rtimes_r G).\]
In order for the left hand sides to match, we expect an isomorphism
	\[H^G_*(X,\K_A^G)\overset{!}\cong KK_*^G(C_0(X),A)\]
for all cocompact proper $G$-spaces $X$. For a finite subgroup $H\subseteq G$ and $X=G/H$, this boils down to the isomorphism
	\[H^G_*(G/H,\K_A^G)\overset{!}\cong KK_*^G(C_0(G/H),A)\cong KK_*^H(\C,A)\cong K_*(A\rtimes_r H).\]
	Now we are tempted to define $\K_A^G(G/H):=\K(A\rtimes_r H)$, where $\K:\mathfrak C^*\to \Sp$ is a functor representing $K$-theory for $C^*$-algebras. Unfortunately, the assignment $G/H\mapsto A\rtimes_r H$ does not define a \emph{functor} on the orbit category. To solve this, we replace $A\rtimes_r H$ by a Morita-equivalent $C^*$-\emph{category} $A\rtimes_r \overline{G/H}$. We then define $\K_A^G (G/H)$ to be the $K$-theory spectrum in the sense of \cite{joachim2003k} of that $C^*$-category. The construction of the $C^*$-category $A\rtimes_r \overline{G/H}$ given here is a minor modification of the construction in \cite{mitchener2002c}. 

\subsection*{Groupoid actions and crossed products}

\begin{defn} \label{ccategory}
		A \emph{unital $C^*$-category} is a small category $\mathcal A$, whose morphism sets $\mathcal A(x,y)$ are complex Banach spaces equipped with conjugate linear \emph{involution maps}
			$*:\mathcal A(x,y)\to \mathcal A(y,x)$
		satisfying the axioms 
			\begin{enumerate}				
				\item $(a^*)^*=a$
				\item $\|a b\|\leq \|a\|\|b\|$
				\item $\|a^* a\|=\|a\|^2$\label{C*identity}
				\item $(a b)^*=b^*a^*$
				\item $a^* a\geq 0$ \label{positive-elements}
			\end{enumerate}
		for all morphisms $a\in \mathcal A(y,z), b\in \mathcal A(x,y)$. A \emph{unital $C^*$-functor} is a functor between $C^*$-categories which is linear on morphism sets and preserves the involution. A (nonunital) \emph{$C^*$-category} is defined in the same way as a unital $C^*$-category except that the morphism sets are not required to contain identity morphisms. A (nonunital) \emph{$C^*$-functor} is defined in the same way as a $C^*$-functor except that it does not need to preserve identity morphisms. By dropping the norm from the definition, we obtain analog notions of $*$-categories and $*$-functors. 
	\end{defn}
	
	\begin{defn}
		A \emph{groupoid} $\mathcal G$ is a small category with all morphisms invertible. We do not equip groupoids with any topology. A \emph{groupoid morphism} $F:\mathcal G\to \mathcal H$ is a functor between the underlying categories. A \emph{$\mathcal G$-$C^*$-algebra} $A$ is a functor $x\mapsto A_x$ from $\mathcal G$ to the category of $C^*$-algebras. A \emph{$\mathcal G$-equivariant morphism} $A\to B$ is a natural transformation of the underlying functors. 
	\end{defn}
	
	Sticking to the notation for $G$-$C^*$-algebras, we denote the action of an element $g\in \mathcal G(x,y)$ by $\alpha_g:A_x\to A_y$ and say that \emph{the $\mathcal G$-action is denoted by $\alpha$.}	
	
	\begin{rem}
		Our definition of $\mathcal G$-$C^*$-algebras is adapted from \cite{mitchener2002c} and formally differs from the classical definition (e.g. the one in \cite{legall1999theorie}). Usually, a $\mathcal G$-$C^*$-algebra $A$ is defined as a single $C^*$-algebra $A$ together with a non-degenerate $*$-homomorphism $\varphi:C_0(\Ob(\mathcal G))\to ZM(A)$ and an additional datum implementing the action. Our definition can be obtained from the classical one by taking the fibers 
			\[A_x:=A/\varphi(C_0(\Ob(\mathcal G)\setminus \{x\})A).\] 
	\end{rem}
	
	\begin{eg}\label{transformationgroupoids}
		Let $G$ be a discrete group acting on a set $X$. The \emph{transformation groupoid} $\overline X$ has the points of $X$ as objects and morphisms given by 
			\[\overline X(x,y):=\{g\in G\mid gx=y\}.\]
		Every $G$-equivariant map $X\to Y$ gives rise to a faithful (i.e. injective on morphism sets) groupoid morphism $\overline X\to \overline Y$. In particular, there is a natural morphism $\overline X\to G=\overline {\pt}$. By precomposition with this morphism, every $G$-$C^*$-algebra (considered as a functor from $G$ to the category of $C^*$-algebras) can be considered as an $\overline X$-$C^*$-algebra as well.
	\end{eg}
	
	\begin{defn}\label{convolutioncategory}
		Let $\mathcal G$ be a groupoid and $A$ a $\mathcal G$-$C^*$-algebra with $\mathcal G$-action denoted by $\alpha$. The \emph{convolution category} $A\mathcal G$ is the category with the same objects as $\mathcal G$ and morphism sets given by formal sums
			\[{A\mathcal G}(x,y):=\left\{\sum_{i=1}^n a_iu_{g_i}\biggm| n\in \N,g_i\in {\mathcal G}(x,y),a_i\in A_y\right\}.\]
		We define composition and involution on $A\mathcal G$ by linear extension of the formulas
			\[au_g\cdot bu_h:=a\alpha_g(b)u_{gh},\quad (au_g)^*:=\alpha_{g^{-1}}(a)^*u_{g^{-1}}\]
		for $a\in A_z,b\in A_y,h\in \mathcal G(x,y)$ and $g\in \mathcal G(y,z)$. In this way, $A\mathcal G$ becomes a $*$-category.
		
	\end{defn}
	\begin{defn}\label{groupoid-regular-representation}

		Let $A$ be a $\mathcal G$-$C^*$-algebra with $\mathcal G$-action $\alpha$. Let $x,y\in \Ob(\mathcal G)$ and choose $z\in \Ob(\mathcal G)$ such that $\mathcal G(z,x)$ is nonempty. To each $f\in A\mathcal G(x,y)$, we associate an adjointable operator 
			\[\Lambda_{A,\mathcal G,z}(f):\ell^2(\mathcal G(z,x),A_z)\to \ell^2(\mathcal G(z,y),A_z)\]
		of Hilbert-$A_z$-modules, defined by linear extension of the formula 
			\[\Lambda_{A,\mathcal G,z}(au_g)\xi(h):=\alpha_{h^{-1}}(a)\xi(g^{-1}h)\]
			for $a\in A_y,\xi\in \ell^2(\mathcal G(z,x),A_z), g\in \mathcal G(x,y)$ and $h\in \mathcal G(z,y)$.
		The \emph{reduced norm of $f$} is given by
			\begin{equation}\label{reducednorm}
				\|f\|_r:=\|\Lambda_{A,\mathcal G,z}(f)\|.
			\end{equation}
		The \emph{reduced crossed product $A\rtimes_r \mathcal G$} is the $C^*$-category obtained from $A\mathcal G$ by completing all the morphism sets with respect to the reduced norm. 
	\end{defn}
	
	\begin{rem}
		The norm in \eqref{reducednorm} does not depend on the choice of $z$. Indeed, if $z'\in \Ob(\mathcal G)$ is another object such that $\mathcal G(x,z')$ is nonempty, we may pick a morphism $g\in \mathcal G(z,z')$. A calculation then shows that for every $f\in A\mathcal G(x,y)$, the diagram
			
			\begin{equation}\label{independentofz}
			\begin{tikzcd}
				\ell^2(\mathcal G(z,x),A_z)\arrow[d,"\Lambda_{A,\mathcal G,z}(f)"]\arrow[r,"\rho_g\otimes \alpha_g","\cong"']	&\ell^2(\mathcal G(z',x),A_{z'})\arrow[d,"\Lambda_{A,\mathcal G,z'}(f)"]\\
				\ell^2(\mathcal G(z,y),A_z)\arrow[r,"\rho_g\otimes \alpha_g","\cong"']	&\ell^2(\mathcal G(z',y),A_{z'})
			\end{tikzcd}
			\end{equation}
		commutes where $\rho_g\otimes \alpha_g$ is defined by the formula 
			\[(\rho_g\otimes \alpha_g)\xi(h)=\alpha_g(\xi(hg)).\]
		Thus, we have $\|\Lambda_{A,\mathcal G,z}(-)\|=\|\Lambda_{A,\mathcal G,z'}(-)\|$.
		
		It is sometimes convenient to have a fixed representation of the reduced crossed product. We call the representation 
			\begin{equation}\label{regularreresentationeq}
				\Lambda_{A,\mathcal G}:=\prod_{z\in \Ob(\mathcal G)}\Lambda_{A,\mathcal G,z}:A\mathcal G\to \prod_{z\in \Ob(\mathcal G)}\mathcal L_{A_z}\left(\bigoplus_{x\in \Ob(\mathcal G)}\ell^2(\mathcal G(z,x),A_z)\right)
			\end{equation}
		the \emph{regular representation of $A\mathcal G$.}
	\end{rem}
	
	\begin{lem}\label{functorial-reduced-crossed-product}
	The following statements hold.
		\begin{enumerate}		
		\item
			 Let $A,B$ be $\mathcal G$-$C^*$-algebras and $\varphi:A\to B$ a $\mathcal G$-equivariant morphism. Then the canonical $*$-functor 
				$\varphi\mathcal G:A\mathcal G\to B\mathcal G$,			
			defined as the identity on objects and as $au_g\mapsto \varphi(a)u_g$ on morphisms, extends to a $C^*$-functor 
				\[\varphi\rtimes_r\mathcal G:A\rtimes_r\mathcal G\to B\rtimes_r \mathcal G.\]			
		\item 
			Let $A$ be a $\mathcal G$-$C^*$-algebra and $\varphi:\mathcal H\to \mathcal G$ a faithful groupoid morphism. Denote the $\mathcal H$-$C^*$-algebra obtained by precomposition with $\varphi$ also by $A$. Then the natural $*$-functor 
				$\id_A\varphi:A\mathcal H\to A\mathcal G$,
			defined by $x\mapsto \varphi(x)$ on objects and $au_g\mapsto au_{\varphi(g)}$ on morphisms extends to an isometric $C^*$-functor 
				\[\id_A\rtimes_r \varphi:A\rtimes_r \mathcal H\to A\rtimes_r \mathcal G.\]
			\end{enumerate}
			\begin{proof}
			
				For the first statement, fix $x,y\in \Ob(\mathcal G)$ and fix $z\in \Ob(\mathcal G)$ such that $\mathcal G(z,x)$ is nonempty. Consider the following commutative diagram.
					\[
					\begin{tikzcd}
					A\mathcal G(x,y)\arrow[d,"\varphi\mathcal G"] \arrow[r,"\Lambda_{A,\mathcal G,z}"]	
					&\mathcal L_{A_z}\left(\bigoplus_{w\in \Ob(\mathcal G)} \ell^2(\mathcal G(z,w),A_z)\right)\arrow[r,"\cong"]	
					&M(A_z\otimes \mathcal K(\bigoplus_{w\in \Ob(\mathcal G)} \ell^2\mathcal G(z,w)))\arrow[d,"\varphi_z\otimes \id",dashed]\\	
					B\mathcal G(x,y) \arrow[r,"\Lambda_{B,\mathcal G,z}"]	
					&\mathcal L_{B_z}\left(\bigoplus_{w\in \Ob(\mathcal G)} \ell^2(\mathcal G(z,w),B_z)\right)\arrow[r,"\cong"]	
					&M(B_z\otimes \mathcal K(\bigoplus_{w\in \Ob(\mathcal G)} \ell^2\mathcal G(z,w)))
					\end{tikzcd}.
					\]

				The horizontal isomorphisms are the standard identifications (cp. \cite[Thm. 2.4 and p. 37]{lance1995hilbert}). 
				In general, $\varphi_z\otimes \id$ does not extend to the whole multiplier algebra. However, it extends to a $C^*$-subalgebra which contains the image of $\Lambda_{A,\mathcal G,z}$ by \cite[Def. A.3, Prop. A.6 (i)]{echterhoff2002categorical}. In any case, the extension of $\varphi_z\otimes \id$ is norm-decreasing. Since the horizontal arrows in the above diagram are isometric by definition, $\varphi \mathcal G$ must be norm-decreasing as well. 
				
				For the second statement of the lemma, fix $x,y\in \Ob(\mathcal H)$ and pick $z\in \Ob(\mathcal H)$ such that $\mathcal H(z,x)$ is nonempty. We have to prove the following equation.
					\begin{equation}\label{repsofHandG}
						\|\Lambda_{A,\mathcal H,z}(f)\|=\|\Lambda_{A,\mathcal G,\varphi(z)}\circ(\id_A\varphi)(f)\|,\quad \text{for all} f\in A\mathcal H(x,y).
					\end{equation}
				Let $S\subseteq \mathcal G(\varphi(z),\varphi(z))$ be a system of coset representatives for $\mathcal H(z,z)\backslash \mathcal G(\varphi(z),\varphi(z))$ (this expression makes sense since $\varphi$ is faithful). We get a direct sum decomposition
					\[\ell^2(\mathcal G(\varphi(z),\varphi(x)),A_{\varphi(z)})=\bigoplus_{g\in S}\ell^2(\mathcal H(z,x)g,A_{\varphi(z)})\]
				and a similar decomposition for $y$ instead of $x$. As in \eqref{independentofz}, we have a commutative diagram
					\[\begin{tikzcd}
						\ell^2(\mathcal H(z,x)g,A_{\varphi(z)})\arrow[d,"\Lambda_{A,\mathcal H,z}(f)"]\arrow[r,"\rho_g\otimes \alpha_g","\cong"']	&\ell^2(\mathcal H(z,x),A_{\varphi(z)})\arrow[d,"\Lambda_{A,\mathcal H,z}(f)"]\\
						\ell^2(\mathcal H(z,y)g,A_{\varphi(z)})\arrow[r,"\rho_g\otimes \alpha_g","\cong"']	&\ell^2(\mathcal H(z,y),A_{\varphi(z)})
					\end{tikzcd}\]
				for every $f\in A\mathcal H(x,y)$ and $g\in S$. Thus, the representation $\Lambda_{A,\mathcal G,\varphi(z)}\circ (\id_A\varphi)$ of $A\mathcal H(x,y)$ is equivalent to a direct sum of $|S|$-many copies of the representation $\Lambda_{A,\mathcal H,z}$. This proves \eqref{repsofHandG}.
							
			\end{proof}
	\end{lem}

\subsection*{$C^*$-algebras associated to $C^*$-categories}	
	We now recall the construction from \cite{joachim2003k} of a $K$-theory spectrum for $C^*$-categories. The idea is to first associate a $C^*$-algebra to a $C^*$-category and then associate a $K$-theory spectrum to this $C^*$-algebra. There are two $KK$-equivalent constructions of the associated $C^*$-algebra. The first construction is easier to compute for our examples while the second construction has better functoriality properties.

	\begin{defn}[\cite{joachim2003k}]
		Let $\mathcal A$ be a $C^*$-category. We equip 
			\[\mathfrak C^*_0\mathcal A:=\bigoplus_{x,y}\mathcal A(x,y)\]
		with the structure of a $*$-algebra by inheriting the involution from $\mathcal A$ and by defining the product of two elements $f\in \mathcal A(x,y),g\in \mathcal A(z,w)$ to be
			\[g\cdot f:=\begin{cases}gf,& y=z \\ 0,&y\not=z	\end{cases}.\]

		We denote by $\mathfrak C^*\mathcal A$ the enveloping $C^*$-algebra of $\mathfrak C^*_0\mathcal A$, i.e. the completion with respect to the supremum of all $C^*$-semi-norms.

	\end{defn}
	
	In \cite{joachim2003k}, the above $C^*$-algebra is denoted by $A_\mathcal A$ rather than $\mathfrak C^*\mathcal A$. 	
	
	\begin{rem}
	\begin{enumerate}	
	\item The supremum of all $C^*$-semi-norms $\rho$ on $\mathfrak C^*_0\mathcal A$ is indeed finite: The semi-norm of an element $a=\sum_{x,y} a_{xy}\in \mathfrak C^*_0\mathcal A$ with $a_{xy}\in \mathcal A(x,y)$ can be  bounded by
			\[\rho(a)\leq \sum_{x,y}\rho(a_{xy})=\sum_{x,y}\rho(a_{xy}^*a_{xy})^{\frac 1 2}\leq \sum_{x,y}\|a_{xy}^*a_{xy}\|^{\frac 1 2}.\]
	since each $C^*$-algebra $\mathcal A(x,x)$ has a unique $C^*$-norm. Furthermore it is shown in \cite{joachim2003k} that the supremum is indeed a \emph{norm}.

	\item The $C^*$-category $\mathfrak C^*\mathcal A$ has the following universal property: Given any $C^*$-algebra $B$ and any $C^*$-functor $F:\mathcal A\to B$ satisfying $F(f)F(g)=0$ for all non-composable morphisms $f$ and $g$, there is a unique $*$-homomorphism $\mathfrak C^*F:\mathfrak C^*\mathcal A\to B$ such that the following diagram commutes:
		\[\begin{tikzcd}
			\mathcal A\arrow[d]\arrow[r,"F"]&B\\
			\mathfrak C^*\mathcal A\arrow[ur,"\mathfrak C^*F"']&
		\end{tikzcd}\]
	\end{enumerate}
	\end{rem}

Unfortunately, the assignment $\mathcal A\mapsto \mathfrak C^* \mathcal A$ is not functorial with respect to arbitrary $C^*$-functors. Before giving a functorial construction, we list some useful properties of $\mathfrak C^*\mathcal A$. 

\begin{lem}	\label{isometriconmorphisms}
	Let $\mathcal A$ be a $C^*$-category, $B$ a $C^*$-algebra and $F:\mathcal A\to B$ a $C^*$-functor satisfying $F(f)F(g)=0$ whenever $f$ and $g$ are non-composable morphisms in $\mathcal A$. Suppose that $F$ is isometric on morphism sets and that 
		\[\mathfrak C^*_0F:\mathfrak C^*_0\mathcal A\to B\]
	is injective. Then 
		$\mathfrak C^*F:\mathfrak C^*\mathcal A\to B$
	is isometric. 
	\begin{proof}
		The proof is a variant of the proof of \cite[Lem. 3.6]{joachim2003k}. 
		We have to show that $\mathfrak C^*_0F$ is isometric. By construction, we have
			\[\mathfrak C^*_0\mathcal A=\bigcup_{\mathcal A'}\mathfrak C^*_0\mathcal A',\]
		where the union is taken over all full subcategories $\mathcal A'\subseteq \mathcal A$ with only finitely many objects. It suffices to show, that $\mathfrak C^*_0F$ is isometric on each $\mathfrak C^*_0\mathcal A'$. Since $F$ is isometric and $\mathfrak C^*_0F$ is injective, it suffices to show that there is only one $C^*$-norm $\|\cdot \|$ on $\mathfrak C^*_0\mathcal A'$ which restricts to the given norm on the morphism sets (note that the inclusions $\mathcal A'(x,y)\hookrightarrow \mathfrak C^*\mathcal A'$ are isometric since $F$ is isometric and $\mathfrak C^*F$ norm-decreasing). Write $a\in \mathfrak C^*_0\mathcal A'$ as a finite sum
			\[a=\sum a_{xy},\quad a_{xy}\in \mathcal A'(x,y)\]
		and denote by $N$ the number of objects of $\mathcal A'$. Then the estimate 
			\begin{equation}\label{alreadycomplete}
				\max_{x,y} \|a_{xy}\|\leq \|a\|\leq N^2\max_{x,y} \|a_{xy}\|
			\end{equation}
		shows that $\|\cdot \|$ is already complete on $\mathfrak C^*_0\mathcal A'$ and therefore the unique $C^*$-norm with this property. The first inequality in \eqref{alreadycomplete} can be verified by writing $a_{xy}=\lim_{\lambda} u_\lambda av_\lambda$ for approximate units $u_\lambda \in \mathcal A'(y,y)$ and $v_\lambda \in \mathcal A'(x,x)$. 
	\end{proof}
\end{lem}

\begin{cor}\label{classicalgroupoidcrossedproduct}
	Let $A$ be a $\mathcal G$-$C^*$-algebra. Then $\mathfrak C^*(A\rtimes_r \mathcal G)$ is naturally isomorphic to the classical reduced crossed product $C^*$-algebra of $A$ as defined in \cite[Sec. 1.4]{anantharaman2016some}.
	\begin{proof}
		Denote the classical reduced crossed product of $A$ by $\widetilde{A\rtimes_r \mathcal G}$. Although using different notation, it is precisely defined as the closed image of the regular representation $\Lambda_{A,\mathcal G}$ from \eqref{regularreresentationeq}.
		Since $\Lambda_{A,\mathcal G}$ is by definition isometric on morphism sets, Lemma \ref{isometriconmorphisms} provides us with an isomorphism
			\[\mathfrak C^*\Lambda_{A,\mathcal G}:\mathfrak C^*(A\rtimes_r \mathcal G)\to \widetilde{A\rtimes_r \mathcal G}.\]
	\end{proof}
\end{cor}

In particular, we obtain the following special case:

\begin{cor}\label{classicalcrossedproduct}
	Let $G$ be a discrete group acting on a set $X$. Let $A$ be a $G$-$C^*$-algebra and consider $A$ as an $\overline X$-$C^*$-algebra as in Example \ref{transformationgroupoids}. Then there is a natural isomorphism
		\[\mathfrak C^*(A\rtimes_r \overline X)\cong C_0(X,A)\rtimes_r G.\]
\end{cor}

\begin{cor}\label{ccact-tensor-products}
	Let $A$ be a $\mathcal G$-$C^*$-algebra and $B$ a $C^*$-algebra (endowed with the trivial $\mathcal G$-action). Then there is a canonical $*$-isomorphism 
		\[\mathfrak C^*((A\otimes B)\rtimes_r \mathcal G)\cong \mathfrak C^*(A\rtimes_r \mathcal G)\otimes B.\]
	
		\begin{proof}
			By Lemma \ref{isometriconmorphisms}, the representation 
				\[\mathfrak C^*\Lambda_{A\otimes B,\mathcal G}:\mathfrak C^*((A\otimes B)\rtimes_r\mathcal G)\to  \prod_z\mathcal L_{A_z\otimes B}(\oplus_x\ell^2(\mathcal G(z,x),A_z\otimes B))\]
			is faithful. Its image coincides with the image of the faithful representation 
				\[\mathfrak C^*(A\rtimes_r\mathcal G)\otimes B\to \prod_z \mathcal L_{A_z}(\oplus_x \ell^2(\mathcal G(z,x),A_z))\otimes B\to \prod_z \mathcal L_{A_z\otimes B}(\oplus_x \ell^2(\mathcal G(z,x),A_z\otimes B)).\]
		\end{proof}
\end{cor}

	\begin{defn}[{\cite[Def. 3.7]{joachim2003k}}]\label{defn-of-universal-algebra}
	Let $\mathcal A$ be a $C^*$-category. We denote by $\mathfrak C^*_f\mathcal A$ the universal $C^*$-algebra generated by symbols $(f)$ for morphisms $f\in \mathcal A(x,y)$ subject to the relations
	
		\[(\lambda f+g)=\lambda (f)+(g), \quad (f^*)=(f)^*, \quad (hg)=(h)(g)\]
		for $f,g\in \mathcal A(x,y),h\in \mathcal A(y,z)$ and $\lambda \in \C$.
	By construction, $\mathcal A\mapsto \mathfrak C^*_f\mathcal A$ is the left adjoint functor of the inclusion functor from the category of $C^*$-categories to the category of $C^*$-algebras.
\end{defn}

In \cite{joachim2003k}, the above algebra is denoted by $A^f_\mathcal A$ rather than $\mathfrak C^*_f \mathcal A$. 

\begin{prop}[{\cite[Prop. 3.8]{joachim2003k}}]\label{functorialcalgebra}
	Let $\mathcal A$ be a $C^*$-category with countably many objects and separable morphism sets. Then the canonical $*$-homomorphism $\mathfrak C^*_f\mathcal A\to \mathfrak C^*\mathcal A$ is a stable homotopy equivalence and therefore a $KK$-equivalence.
\end{prop}

The reader should not be concerned about the unitality assumptions in \cite{joachim2003k} since they are not used in the proof of the above proposition. 

\subsection*{A $K$-theory spectrum}

We now recall very briefly the construction of the $K$-theory spectrum $\K$ from \cite{joachim2003k}. We use this particular model because it is quite easy to show that $\K$ maps mapping cone sequences to fiber sequences, $KK$-equivalences to stable equivalences, suspensions to loops and direct sums to wedge sums. The definition of $\K$ involves graded $C^*$-algebras. We only recall the basic definitions and refer to \cite{blackadar1998k} for a more detailed account on graded $C^*$-algebras. A graded $C^*$-algebra is a $\Z_2$-$C^*$-algebra, i.e. a $C^*$-algebra $A$ together with a self-inverse \emph{grading automorphism} $\alpha$. We will need the following examples:

\begin{enumerate}	
	\item We denote by $\hat {\mathcal K}$ the graded $C^*$-algebra of compact operators on $\ell^2\N\oplus \ell^2\N$ with grading automorphism given by conjugation with the unitary $\begin{pmatrix}0&1\\ 1&0\end{pmatrix}$. 
	\item We denote by $\hat {\mathcal S}$ the $C^*$-algebra $C_0(\R)$ with grading automorphism given by reflecting functions at the origin $0\in \R$.
	\item The \emph{Clifford algebra} $\C_n$ on $n$ generators is the universal $C^*$-algebra generated by selfadjoint unitaries $e_1,\dotsc,e_n$ satisfying
		$e_ie_j=-e_je_i,\quad i\not=j$
	The grading automorphism of $\C_n$ is given by $e_i\mapsto -e_i$. 
	\item If not specified otherwise, we equip any $C^*$-algebra with a trivial grading. 
\end{enumerate}

For any two graded $C^*$-algebras $A$ and $B$, there is a \emph{spatial graded tensor product} $A\hat \otimes B$. It is a completion of the algebraic tensor product $A\odot B$ equipped with a non-standard multiplication and involution depending on the grading \cite[Def. 14.4.1]{blackadar1998k}. If one of the factors is trivially graded, then $A\hat \otimes B$ is naturally isomorphic to the usual spatial tensor product. We denote the space of $\Z_2$-equivariant $*$-homomorphisms $A\to B$ by $\mathfrak C^*_{\Z_2}(A,B)$ and endow it with the compact-open topology and the zero morphism as a basepoint. 

\begin{prop}[{\cite[Prop. 4.1]{joachim2003k}}]\label{mapping-loops}
	Let $A,B$ be graded $C^*$-algebras and $X$ a locally compact space. Then there is a natural homeomorphism
		\[\grc(A,B\hat \otimes C_0(X))\to \Top_*(X^+, \grc(A,B)),\quad f\mapsto (x\mapsto (\id_B\hat\otimes\ev_x)\circ f),\]	
	where $X^+$ denotes the one-point compactification and $\ev_x:C_0(X)\to \C$ the evaluation at $x$.
\end{prop}

\begin{defn}
	Let $A$ be a separable $C^*$-algebra. The spectrum $\K(A)$ is given by the sequence of pointed spaces 
		\[\K(A)_n:=\mathfrak C^*_{\Z_2}(\hat {\mathcal S},A\hat \otimes \C_n\hat \otimes \hat {\mathcal K})\]
	and structure maps $\K(A)_n\to \Omega\K(A)_{n+1}$ given by
		\[\mathfrak C^*_{\Z_2}(\hat {\mathcal S},A\hat \otimes \C_n\hat \otimes \hat {\mathcal K})\xrightarrow[\sim]{\beta}\mathfrak C^*_{\Z_2}(\hat {\mathcal S},A\hat \otimes \C_{n+1}\hat \otimes C_0((0,1))\hat \otimes \hat {\mathcal K})\overset{\ref{mapping-loops}}\cong \Omega\mathfrak C^*_{\Z_2}(\hat {\mathcal S},A\hat \otimes \C_{n+1}\hat \otimes \hat {\mathcal K}),\]
	where $\beta$ denotes the Bott map from \cite[Lecture 1]{higson2004group}.
	
	Let $\mathcal A$ be a $C^*$-category with countably many objects and separable morphism sets. The $K$-theory spectrum of $\mathcal A$ is given by 
		\[\K(\mathcal A):=\K(\mathfrak C^*_f\mathcal A)\simeq \K(\mathfrak C^*\mathcal A).\]
\end{defn}

\begin{prop}[{\cite[Thm. A.2]{joachim2003k}}]
	Let $A$ be a trivially graded separable $C^*$-algebra. Then there is a natural isomorphism 
		$\pi_*\K(A)\cong K_*(A)$.
\end{prop}

\begin{defn}\label{defnofKAG}
	Let $G$ be a countable discrete group and $A$ a separable $G$-$C^*$-algebra. We define an $\Or(G)$-spectrum $\K_A^G$ by
		\[\K_A^G(G/H):=\K(A\rtimes_r \overline{G/H}).\]
	The \emph{Davis--L\"uck assembly map for $G$ with coefficients in $A$} is the map 
		\[H^G_*(\mathcal E_{\Fin}G,\K_A^G)\to H^G_*(\pt,\K_A^G)\]
	where $\Fin$ denotes the family of finite subgroups. 
\end{defn}

Note that the functoriality of $\mathfrak C^*_f$ and Lemma \ref{functorial-reduced-crossed-product} guarantee functoriality of $\K_A^G(G/H)$ both in $G/H$ for $G$-equivariant maps and in $A$ for $G$-equivariant $*$-homomorphisms. 

\begin{lem}\label{K is triangulated}
	The functor $A\mapsto \K_A^G$ from the category of separable $G$-$C^*$-algebras to the category of $\Or(G)$-spectra has the following properties:
	\begin{enumerate}
		\item It maps $KK^G$-equivalences to stable equivalences.
		\item It maps mapping cone sequences to fiber sequences.
		\item For any separable $G$-$C^*$-algebra $A$, we have $\K_{SA}^G\cong \Omega \K_A^G$.
		\item Let $A_i,i\in I$ be a countable family of separable $G$-$C^*$-algebras. Then there is a natural stable equivalence 
			$\bigvee_{i\in I}\K_{A_i}^G\simeq \K_{\oplus_{i\in I} A_i}^G$.
	\end{enumerate}
	\begin{proof}
		We fix $G/H\in \Or(G)$ throughout the proof.
		\begin{enumerate}
			\item Every $KK^G$-equivalence $A\to B$ induces a $KK$-equivalence \[C_0(G/H,A)\rtimes_r G\to C_0(G/H,B)\rtimes_r G.\] Therefore, the induced map 
				\begin{multline*}
					\hspace{2cm}\pi_*(\K_A^G(G/H))\overset{\ref{classicalcrossedproduct}}\cong K_*(C_0(G/H,A)\rtimes_r G)\\
					\to K_*(C_0(G/H,B)\rtimes_r G)\overset{\ref{classicalcrossedproduct}}\cong \pi_*(\K_B^G(G/H))\hspace{2cm}
				\end{multline*}
				is an isomorphism.
				
			\item We claim that the functor $A\mapsto \mathfrak C^*(A\rtimes_r\overline {G/H})$ preserves mapping cone sequences and that the functor $A\mapsto \K(A)$ maps mapping cone sequences to fiber sequences. Let $\Cone(\pi)\to A\xrightarrow{\pi} B$ be a mapping cone sequence. For the first claim, use Corollaries \ref{classicalcrossedproduct} and \ref{ccact-tensor-products} to identify $\mathfrak C^*(\Cone(\pi)\rtimes_r\overline{G/H})$ with the cone of the map
				$\mathfrak C^*(A\rtimes_r \overline {G/H})\to \mathfrak C^*(B\rtimes_r \overline{G/H})$.	
			For the second claim, use Proposition \ref{mapping-loops} to identify $\K(\Cone(\pi))$ with the homotopy fiber of the map 
				$\K(A)\to \K(B)$.

			\item By Corollary \ref{ccact-tensor-products}, the functor $A\mapsto \mathfrak C^*(A\rtimes_r\overline{G/H})$ commutes with suspensions. Now the claim follows from Proposition \ref{mapping-loops}.
			
			\item There is a natural map 
				$\bigvee_{i\in I}\K_{A_i}^G(G/H)\to \K_{\oplus_{i\in I}A_i}^G(G/H)$
				which we claim to be a stable equivalence. On homotopy groups, the above map can be written as the composition
				\begin{align*}
					&\pi_*\left(\bigvee_{i\in I}\K(\mathfrak C^*(A\rtimes_r\overline {G/H}))\right) \xrightarrow{\cong}\bigoplus_{i\in I}\pi_*\K(\mathfrak C^*(A_i\rtimes_r\overline{G/H}))			\\
					&\xrightarrow{\cong}\pi_*\K\left(\bigoplus_{i\in I}\mathfrak C^*(A_i\rtimes_r \overline{G/H})\right)\xrightarrow{\cong}\pi_*\K\left(\mathfrak C^*\left(\left(\bigoplus_{i\in I}A_i\right)\rtimes_r \overline {G/H}\right)\right).		
				\end{align*}
				The first map is an isomorphism by Lemma \ref{homotopydirectsums}, the second map is an isomorphism since $K$-theory commutes with countable direct sums \cite[Prop. 6.2.9]{wegge1993k} and the third isomorphism arises from an isomorphism of the underlying $C^*$-algebras. 
		\end{enumerate}
	\end{proof}
\end{lem}

\section{Identification of the assembly maps}\label{identification}
In this section we finally identify the Davis--L\"uck assembly map
	\[H^G_*(\mathcal E_{\Fin}G,\K_A^G)\to H^G_*(\pt,\K_A^G)\]
with the Meyer--Nest assembly map
	\[K_*(\tilde A\rtimes_r G)\to K_*(A\rtimes_r G).\]
The strategy is to use the Dirac morphism $D\in KK^G(\tilde A,A)$ from Theorem \ref{Meyernesttheorem} to compare the Davis--L\"uck map with coefficients in $A$ to the Davis--L\"uck map with coefficients in $\tilde A$. To do so, we need the functor $A\mapsto \K_A^G$ to extend from the category of separable $G$-$C^*$-algebras to the category $\mathfrak{KK}^G$. Due to the choice of our specific model of the $K$-theory spectrum $\K$, it is not obvious how to construct such an extension. One solution could be to choose a $KK$-functorial model for $\K$ which also satisfies Lemma \ref{K is triangulated} and show that the functor $A\mapsto \mathfrak C^*(A\rtimes_r\overline{G/H})$ extends to a triangulated functor $\mathfrak{KK}^G\to \mathfrak{KK}$. The necessary machinery for such an approach can be found in the recent preprint \cite{bunke2021stable} which appeared after the first preprint of this work. However, the author decided to stick to the explicit $K$-theory spectrum and give an elementary solution to the functoriality problem using zig-zags.

\begin{defn}
A \emph{zig-zag of $G$-equivariant $*$-homomorphisms} is a diagram 
		\[A_1\xrightarrow{\varphi_1}B_1\xleftarrow{\psi_1}A_2\xrightarrow{\varphi_2}\dotsb\xrightarrow{\varphi_{n}}B_n\xleftarrow{\psi_n}A_{n+1}.\]	
	of $G$-equivariant $*$-homomorphisms, such that each $\psi_k$ is a $KK^G$-equivalence. Such a zig-zag naturally defines a $KK^G$-class 
		\[[\psi_n]^{-1}\circ[\varphi_n]\circ\dotsb\circ[\psi_1]^{-1}\circ [\varphi_1]\in KK^G(A_1,A_{n+1}).\]	
	 Similarly, a \emph{zig-zag of ($\Or(G)$-)spectra} is a diagram 
		\[\bm E_1\xrightarrow{f_1}\bm F_1\xleftarrow{g_1}\bm E_2\xrightarrow{f_2}\dotsb\xrightarrow{f_{n}}\bm F_n\xleftarrow{g_n}\bm E_{n+1}\]	
	of ($\Or(G)$-)spectra, such that each $g_k$ is a stable equivalence. By Lemma \ref{homotopyinvariance}, every such zig-zag gives rise to a well-defined natural transformation
		\[(\id_X\otimes g_n)_*^{-1}\circ \dotsb\circ(\id_X\otimes f_1)_*:H^G_*(X,\bm E_1)\to H^G_*(X,\bm E_{n+1}).\]
		on homology. Thus, any zig-zag of $G$-equivariant $*$-homomorphisms gives rise to a natural transformation on homology. The following lemma shows, that we can always restrict to the case of zig-zags:
\end{defn}

\begin{lem}\label{kkbymorphisms}
	Every morphism in $\mathfrak{KK}^G$ can be represented by a zig-zag of $G$-equivariant $*$-homomorphisms. 
	\begin{proof}
	This follows from the proofs of \cite[Prop. 6.1, Thm. 6.5]{meyer2000equivariant}.
	\end{proof}
\end{lem}

From now on, we pretend that all $KK^G$-classes are represented by $G$-equivariant $*$-homomorphisms. We leave it to the reader to replace the relevant maps of spectra by zig-zags. Recall from Remark \ref{orGextension} that there is a natural isomorphism $H^G_*(\pt,\K_A^G)\cong K_*(A\rtimes_r G)$. Thus, the Meyer--Nest assembly map can be identified with the map
	$H^G_*(\pt,\K_{\tilde A}^G)\to H^G_*(\pt,\K_A^G)$.
We are now ready to state our main theorem:

\begin{thm}\label{mainidentification}
	The indicated maps in the following diagram are isomorphisms.
		
		\[\begin{tikzcd}
				&H^G_*(\mathcal E_{\Fin}G,\K_{\tilde A}^G)\arrow[r,"\operatorname{pr}_*","\cong"']\arrow[d,"D_*","\cong"']	&H_*^G(\pt,\K_{\tilde A}^G)\arrow[d,"D_*"]\\
				&H_*^G(\mathcal E_{\Fin}G,\K_A^G)\arrow[r,"\operatorname{pr}_*"]&H^G_*(\pt,\K_A^G)
		\end{tikzcd}\]
	In particular, the Meyer--Nest assembly map can be identified with the Davis--L\"uck assembly map. 
\end{thm}

We reduce the proof to the trivial case of finite groups by a series of lemmas. A key ingredient is the following classical theorem. A simple proof of it for discrete groups can be found in \cite[Prop. 6.8]{echterhoff2017crossed}.

\begin{thm}[Green's imprimitivity theorem]\label{greensimprimitivitytheorem}
	Let $G$ be a countable discrete group, $H$ a subgroup and $B$ a separable $H$-$C^*$-algebra with $H$-action $\beta$. Then the $H$-equivariant $*$-homomorphism
	\begin{equation}\label{psi-of-B}
	\psi_B:B\to \Ind_H^GB,\quad b\mapsto \left(g\mapsto \begin{cases}
		\beta_{g^{-1}}(b),&g\in H\\
		0,&g\notin H
	\end{cases}\right)
	\end{equation} 
	gives rise to an inclusion 
		\[B\rtimes_{r}H\xrightarrow{\psi_B\rtimes_{r}H}(\Ind_H^GB)\rtimes_{r}H\xhookrightarrow{}(\Ind_H^GB)\rtimes_{r}G\]
	\text{whose} image is a full corner. In particular, the inclusion 
		\[B\rtimes_r H\hookrightarrow (\Ind_H^GB)\rtimes_r G\] is a $KK$-equivalence. 
\end{thm}

\begin{lem}\label{top-dirac-going-down}
	The map 
		\[D_*:H^G_*(\mathcal E_{\Fin}G,\K_{\tilde A}^G)\to H^G_*(\mathcal E_{\Fin}G,\K_A^G)\]
	is an isomorphism.
	\begin{proof}
		Let $H\subseteq G$ be a finite subgroup and consider the following commutative diagram:
			\[\begin{tikzcd}
				H^G_*(G/H,\K_{\tilde A}^G)\arrow[r]\arrow[d, "\cong"]	&H^G_*(G/H,\K_A^G)\arrow[d,"\cong"]\\
				K_*(C_0(G/H,\tilde A)\rtimes_r G)\arrow[r]	&K_*(C_0(G/H,A)\rtimes_r G)\\
				K_*(\tilde A\rtimes_r H)\arrow[r,"\cong"]\arrow[u,"\cong "]	&K_*(A\rtimes_r H)\arrow[u,"\cong"]
			\end{tikzcd}\]
		Here the horizontal maps are induced by $D$. The vertical isomorphisms are obtained from Corollary \ref{classicalcrossedproduct} and Theorem \ref{greensimprimitivitytheorem}. The lower horizontal map is an isomorphism since $D$ is a $KK^H$-equivalence. Thus, the upper horizontal map is an isomorphism. Since $H\subseteq G$ was an arbitrary finite subgroup, it follows from an excision argument that the map 
			\[D_*:H^G_*(\mathcal E_{\Fin}G,\K_{\tilde A}^G)\to H^G_*(\mathcal E_{\Fin}G,\K_A^G)\]
		is an isomorphism too. 
	\end{proof}
\end{lem}

\begin{lem}\label{consider-induces-algebras}
	Let $\mathcal D\subseteq \mathfrak {KK}^G$ be the full subcategory of all $G$-$C^*$-algebras, for which the Davis--L\"uck assembly map is an isomorphism. Then $\mathcal D$ is localizing in the sense of Definition \ref{localizingdefinition}. 
\end{lem}
In particular, we can reduce the proof of Theorem \ref{mainidentification} to the case $\tilde A=\Ind_H^GB$ for a finite subgroup $H\subseteq G$ and a separable $H$-$C^*$-algebra $B$.

\begin{proof}
	By Lemma \ref{K is triangulated} (i) and Lemma \ref{homotopyinvariance}, $\mathcal D$ is closed under $KK^G$-equivalence. By Lemma \ref{K is triangulated} (iii), the Davis--L\"uck map for a suspension can be identified with the Davis--L\"uck map for the original algebra with homology groups shifted by one. Thus, $\mathcal D$ is closed under suspension. By Lemma \ref{homotopydirectsums}, Lemma \ref{K is triangulated} (iv) and compatibility of the balanced smash product $\wedge_{\Or(G)}$ with wedge sums, the Davis--L\"uck map for a countable direct sum $\bigoplus_{i\in I}A_i$ can be identified with the direct sum of the Davis--L\"uck maps for the individual summands $A_i,i\in I$. Thus, $\mathcal D$ is closed under countable direct sums. It remains to verify stability under mapping cone sequences. Let $\Cone(\pi)\to A\xrightarrow{\pi}B$ be a mapping cone sequence. By Lemma \ref{K is triangulated} (ii), the sequence
		\[\K_{\Cone(\pi)}^G\to \K_A^G\to \K_B^G\]
	is a fiber sequence. Now if $X$ is any $G$-$CW$-complex, the sequence
		\[\Top^G(-,X)_+\wedge_{\Or(G)}\K_{\Cone(\pi)}^G\to \Top^G(-,X)_+\wedge_{\Or(G)}\K_A^G\to \Top^G(-,X)_+\wedge_{\Or(G)}\K_B^G\]
	is still a fiber sequence by Lemmas \ref{fibercofiber} and \ref{cofibertocofiber}. In particular, the rows in the following diagram are exact:
	\[\begin{tikzcd}
							\dotsb\arrow[r] &H^G_*(\mathcal E_{\Fin}G,\K_{\Cone(\pi)}^G)\arrow[r]\arrow[d]&H^G_*(\mathcal E_{\Fin}G,\K_{A}^G)\arrow[r]\arrow[d]&H^G_*(\mathcal E_{\Fin}G,\K_{B}^G)\arrow[r]\arrow[d]&\dotsb \\
							\dotsb\arrow[r] &H^G_*(\mathcal \pt,\K_{\Cone(\pi)}^G)\arrow[r]&H^G_*(\pt,\K_{A}^G)\arrow[r]&H^G_*(\pt,\K_{B}^G)\arrow[r]&\dotsb 
						\end{tikzcd}\]
	It follows from the five-lemma that $\Cone(\pi),A$ and $B$ belong to $\mathcal D$ if at least two of them do. 
\end{proof}

\begin{thm}\label{ind-iso-construction}
	Let $H\subseteq G$ be a finite subgroup, $B$ an $H$-$C^*$-algebra and $X$ a $G$-$CW$-complex. Then there is a natural induction isomorphism
		\[H^H_*(X|_H,\K_B^H)\xrightarrow{\cong}H^G_*(X,\K_{\Ind_H^GB}^G),\]
	where $X|_H$ denotes the restriction of $X$ to $H$. 
		
	\begin{proof}
		Consider the induction functor
			\[I:\Or(H)\to \Or(G),\quad H/K\mapsto G\times_HH/K\cong G/K.\]
		By Lemma \ref{davisinduction}, there is a natural isomorphism 
			\[H^H_*(X|_H,\K_B^H)\cong H^G_*(X,I_*\K_B^H).\]			
		It thus suffices to construct a natural stable equivalence 
			$I_*\K_B^H\simeq \K_{\Ind_H^GB}^G$
		of $\Or(G)$-spectra. We prove this in two steps. Our first claim is that the natural map
			\begin{equation*}\label{plugGKinthespectrum}
				I_*\K_B^H(G/K)=\Top^H(-,G/K)_+\wedge_{\Or(H)}\K_B^H\to \K(B\rtimes_r \overline{(G/K)|_H})
			\end{equation*}
		
		given by $f\wedge x\mapsto f_*(x)$ is a stable equivalence for each $G/K\in \Or(G)$. To see this, decompose $G/K\cong \coprod_i H/L_i$ into $H$-orbits. We get a commutative diagram 
			\[\begin{tikzcd}
				\Top^H(-,G/K)_+\wedge_{\Or(H)}\K_B^H\arrow[r]	&\K(B\rtimes_r \overline{(G/K)|_H})\\
				\bigvee_i \Top^H(-,H/L_i)_+\wedge_{\Or(H)}\K_B^H\arrow[u,"\cong"]\arrow[r,"\cong"]	&\bigvee_i\K(B\rtimes_r\overline{H/L_i})\arrow[u,"\simeq"]
			\end{tikzcd}.\]
		Here the left vertical map is an isomorphism by compatibility of balanced smash products with wedge sums. The lower horizontal map is an isomorphism by Remark \ref{orGextension}. To see that the right-hand map is an equivalence, use Corollary \ref{classicalcrossedproduct} to identify 
			\[\mathfrak C^*(B\rtimes_r \overline{(G/K)|_H})\cong \bigoplus_i\mathfrak C^*(B\rtimes_r\overline{H/L_i})\]
		and apply Lemma \ref{K is triangulated} (iv). This proves the first claim.

		Our second claim is that there is a natural $C^*$-functor 
			\[F:B\rtimes_r \overline{(G/K)|_H}\to (\Ind_H^GB)\rtimes_r \overline{G/K}\]			
		which induces a stable equivalence of $K$-theory spectra. To construct $F$, denote the $H$-action on $B$ by $\beta$ and consider the $H$-equivariant $*$-homomorphism 
			\[\psi_B:B\to \Ind_H^GB,\quad b\mapsto \left(g\mapsto 
				\begin{cases}
					\beta_{g^{-1}}(b),&g\in H\\
					0,&g\notin H
				\end{cases}\right).
			\]		
		Then $\psi_B$ is automatically $\overline{(G/K)|_H}$-equivariant and induces a $C^*$-functor
			\[F:B\rtimes_r\overline{(G/K)|_H}\xrightarrow{\psi_B\rtimes_r \overline{(G/K)|_H}}(\Ind_H^GB)\rtimes_r \overline{(G/K)|_H}\xhookrightarrow{}(\Ind_H^GB)\rtimes_r \overline{G/K}.\]

		To see that $F$ induces a stable equivalence of $K$-theory spectra, we claim that $\mathfrak C^*F$ can be identified with the $*$-homomorphism
			\[C_0(G/K,B)\rtimes_r H\xrightarrow{\psi_{C_0(G/K,B)}\rtimes_r H} \Ind_H^G(C_0(G/K,B))\rtimes_rH\xhookrightarrow{} \Ind_H^G(C_0(G/K, B))\rtimes_r G\]
		from Theorem \ref{greensimprimitivitytheorem}. Indeed this identification can easily be made by using Corollary \ref{classicalcrossedproduct} and checking commutativity of the diagram
		
		\[
				\begin{tikzcd}
					&C_0(G/K,B)\arrow[rr,"\id_{C_0(G/K)}\otimes \psi_B"]\arrow[drr,"\psi_{C_0(G/K,B)}"'] &&C_0(G/K,\Ind_H^GB)\arrow[d,"\alpha","\cong"']\\
					&	&&\Ind_H^G(C_0(G/K,B))					
				\end{tikzcd},
		\]
			where $\alpha$ is defined by 
				\[\alpha(f)(g)(hK):=f(ghK)(g),\quad f\in C_0(G/K,\Ind_H^GB),g\in G,hK\in G/K\]
			as in \cite[Lem. 12.6]{guentner2000equivariant}. This proves the second claim and provides us with a natural equivalence 
			\[I_*\K_B^H(G/K)\simeq \K(B\rtimes_r \overline{(G/K)|_H})\simeq \K_{\Ind_H^GB}^G(G/K),\quad G/K\in \Or(G).\]

	\end{proof}
\end{thm}

\begin{proof}[Proof of Theorem \ref{mainidentification}]
	By Lemma \ref{top-dirac-going-down}, the map 
		\[H^G_*(\mathcal E_{\Fin}G,\K_{\tilde A}^G)\to H^G_*(\mathcal E_{\Fin}G,\K_A^G)\]
	is an isomorphism. By Lemma \ref{consider-induces-algebras}, it suffices to prove that 
		\[H^G_*(\mathcal E_{\Fin}G,\K_{\Ind_H^GB}^G)\to H^G_*(\pt,\K_{\Ind_H^GB}^G)\]		
	is an isomorphism for every finite subgroup $H\subseteq G$ and every $H$-$C^*$-algebra $B$. Theorem \ref{ind-iso-construction} provides us with a commutative diagram 
		\[\begin{tikzcd}
			&H^G_*(\mathcal E_{\Fin}G,\K_{\Ind_H^GB}^G)\arrow[r]	&H^G_*(\pt,\K_{\Ind_H^GB}^G)\\
			&H^H_*(\mathcal E_{\Fin}G|_H,\K_B^H)\arrow[u,"\cong"]\arrow[r,"\cong"]	&H^H_*(\pt,\K_B^H)\arrow[u,"\cong"]
		\end{tikzcd}.\]		
			Since $H$ is finite, $\mathcal E_{\Fin}G$ is $H$-contractible. Thus, the lower map in the diagram is an isomorphism and so is the upper one.
\end{proof}

\section{Exotic crossed products}\label{exotic}

As shown in \cite{higson2002counterexamples}, the Baum--Connes conjecture (with coefficients) turns out to be false in general. The problem is that for certain discrete groups $G$, the functor $A\mapsto K_*(A\rtimes_r G)$ is not exact in the middle. Motivated by these counterexamples, Baum, Guentner and Willett gave a new formulation of the Baum--Connes conjecture in \cite{baum2015expanders} which fixes these counterexamples and is equivalent to the old conjecture for all exact groups. The idea is to replace the \emph{reduced crossed product} by a better behaved crossed product functor. Such \emph{exotic crossed product functors} were also studied extensively by Buss, Echterhoff and Willett, see \cite{buss2016exotic} for a survey.

In this section we reformulate our main result for more general exotic crossed product functors and indicate how to adapt the proofs to this situation. 

\begin{defn}[{\cite[Def. 2.1]{baum2015expanders}}]
	Let $G$ be a countable discrete group. A \emph{crossed product functor} $\rtimes_\mu G$ is a functor $A\mapsto A\rtimes_\mu G$ from the category of $G$-$C^*$-algebras to the category of $C^*$-algebras, such that every $A\rtimes_\mu G$ contains the convolution algebra $AG$ as a dense subalgebra, together with natural transformations 
		\[A\rtimes_{\max}G\to A\rtimes_\mu G\to A\rtimes_r G\]
	which extend the identity on $AG$.
\end{defn}

For a $G$-$C^*$-algebra $A$, the \emph{maximal Meyer--Nest assembly map} can be defined as the map 
	\[K_*(\tilde A\rtimes_{\max} G)\to K_*(A\rtimes_{\max}G)\]
induced by the image of the Dirac morphism $D\in KK^G(\tilde A,A)$ under the descent functor $\rtimes_{\max} G:\mathfrak{KK}^G\to \mathfrak{KK}$. Similarly, there is a \emph{maximal Baum--Connes assembly map}
	\[K_*^G(\mathcal E_{\Fin}G,A)\to K_*(A\rtimes_{\max} G).\]
The maximal Baum--Connes and Meyer--Nest assembly maps can be identified with exactly the same proof as in \cite[Thm. 5.2]{meyer2004baum}.
The reduced versions of the assembly maps can be obtained from the maximal versions by postcomposing with the natural map
	\[K_*(A\rtimes_{\max}G)\to K_*(A\rtimes_r G).\]
More generally, for any crossed product functor $\rtimes_\mu G$, we can define the \emph{$\mu$-Baum--Connes assembly map} and the \emph{$\mu$-Meyer--Nest assembly map} by postcomposing their maximal versions with the natural map 
	\begin{equation}\label{quotientmap}
		K_*(A\rtimes_{\max}G)\to K_*(A\rtimes_\mu G).
	\end{equation}

\begin{rem}
	If an exotic crossed product functor $\rtimes_\mu G$ is Morita-compatible in the sense of \cite[Def. 3.3]{baum2015expanders}, then there is a potentially different way of constructing the $\mu$-Meyer--Nest-assembly map. One could also directly use the descent functor $\rtimes_\mu G:\mathfrak{KK}^G\to \mathfrak{KK}$ constructed in \cite[Prop. 6.1]{buss2018exotic} and define the $\mu$-Meyer--Nest assembly map as the map 
		\[K_*(\tilde A\rtimes_\mu G)\to K_*(A\rtimes_\mu G).\]
	Since the natural transformation $\rtimes_{\max}G\Rightarrow \rtimes_\mu G$ descends to a natural transformation of the corresponding functors on $\mathfrak{KK}^G$, we get a commutative diagram
		\[\begin{tikzcd}
			K_*(\tilde A\rtimes_{\max}G)\arrow[d,"\cong"]\arrow[r]	&K_*(A\rtimes_{\max}G)\arrow[d]\\
			K_*(\tilde A\rtimes_\mu G)\arrow[r]	&K_*(A\rtimes_\mu G)
		\end{tikzcd}\]
	Since the generators of $\langle \mathcal{CI}\rangle$ are proper, the left vertical map is an isomorphism. So in fact we end up with the same assembly map. 
	The same remark also applies to the $\mu$-Baum--Connes assembly map (cp. \cite[Lem. 6.4]{buss2018exotic}).
\end{rem}		
		
To complete the picture, we define a maximal version of the Davis--L\"uck assembly map and indicate how to identify it with the maximal Meyer--Nest assembly map. We can then define the \emph{$\mu$-Davis--L\"uck assembly map} by postcomposing it with the natural map \eqref{quotientmap} and conclude that all three pictures of the $\mu$-assembly map are equivalent. 

\begin{defn}
	Let $\mathcal G$ be a groupoid and $A$ a $\mathcal G$-$C^*$-algebra. Denote by $A\rtimes_{\max}\mathcal G$ the completion of the convolution category $A\mathcal G$ by the supremum of all $C^*$-norms.
\end{defn}
	
\begin{rem}
	The supremum of all $C^*$-semi-norms $\rho$ on $A\mathcal G$ is finite. Indeed, this fact is well-known if $\mathcal G$ is a discrete group and in the general case, the norm of an element $a\in A\mathcal G(x,y)$ can be estimated by 
		\[\rho(a)^2=\rho(a^*a)\leq \|a^*a\|_{A_x\rtimes_{\max}(\mathcal G(x,x))}.\]
\end{rem}

It follows directly from the definition that an analog of Lemma \ref{functorial-reduced-crossed-product} holds for the maximal crossed product, i.e. $A\rtimes_{\max}\mathcal G$ is functorial both in $A$ and in $\mathcal G$. Note that we also defined $\mathfrak C^*(A\rtimes_{\max} \mathcal G)$ as an enveloping $C^*$-algebra. For \emph{discrete} groupoids, the classical maximal groupoid crossed product algebra (defined as in \cite[Sec. 1.4]{anantharaman2016some}) can also be defined as an enveloping $C^*$-algebra of a certain convolution algebra (cp. \cite[Sec. 3]{quigg1999c}). Using this and Lemma \ref{isometriconmorphisms} it is easy to prove an analog of Corollary \ref{classicalgroupoidcrossedproduct}, i.e. that $\mathfrak C^*(A\rtimes_{\max}\mathcal G)$ is canonically isomorphic to the classical maximal groupoid crossed product algebra of the $\mathcal G$-$C^*$-algebra $A$. In particular, we get the following corollary:

\begin{cor}[c.f. Corollary \ref{classicalcrossedproduct}]
	Let $G$ be a discrete group acting on a set $X$. Let $A$ be a $G$-$C^*$-algebra and consider $A$ as an $\overline X$-$C^*$-algebra as in Example \ref{transformationgroupoids}. Then there is a natural isomorphism
		\[\mathfrak C^*(A\rtimes_{\max} \overline X)\cong C_0(X,A)\rtimes_{\max} G.\]	
\end{cor}

An easy application of the universal property of the maximal tensor product also gives the following result:

\begin{lem}[cp. Lemma \ref{ccact-tensor-products}]
	Let $\mathcal G$ be a groupoid, $A$ a $\mathcal G$-$C^*$-algebra and $B$ a $C^*$-algebra (endowed with the trivial $\mathcal G$-action). Then there is a canonical $*$-isomorphism 
		\[\mathfrak C^*((A\otimes_{\max} B)\rtimes_{\max} \mathcal G)\cong \mathfrak C^*(A\rtimes_{\max} \mathcal G)\otimes_{\max} B.\]
\end{lem}

Note that we only applied the above lemma for nuclear $B$, in which case we safely can replace the maximal tensor product by the minimal tensor product. 

\begin{defn}[c.f. Definition \ref{defnofKAG}]
	Let $G$ be a discrete group acting on a $C^*$-algebra $A$. We define an $\Or(G)$-spectrum $\K_{A,\max}^G$ by 
		\[\K_{A,\max}^G(G/H):=\K(A\rtimes_{\max}\overline{G/H}).\]
	The \emph{maximal Davis--L\"uck assembly map} is the map
		\[H^G_*(\mathcal E_{\Fin}G,\K_{A,\max}^G)\to H^G_*(\pt,\K_{A,\max}^G).\]
\end{defn}

From the above, one can adapt the proof of Lemma \ref{K is triangulated} to $\K_{A,\max}^G$:

\begin{lem}[c.f. Lemma \ref{K is triangulated}]
	The functor $A\mapsto \K_{A,\max}^G$ from the category of separable $G$-$C^*$-algebras to the category of $\Or(G)$-spectra has the following properties:
	\begin{enumerate}
		\item It maps $KK^G$-equivalences to stable equivalences.
		\item It maps mapping cone sequences to fiber sequences.
		\item For any separable $G$-$C^*$-algebra $A$, we have $\K_{SA,\max}^G\cong \Omega \K_{A,\max}^G$.
		\item Let $A_i,i\in I$ be a countable family of separable $G$-$C^*$-algebras. Then there is a natural stable equivalence 
			$\bigvee_{i\in I}\K_{A_i,\max}^G\simeq \K_{\oplus_{i\in I} A_i,\max}^G$.
	\end{enumerate}
\end{lem}

Using this and the maximal version of Green's imprimitivity theorem (Theorem \ref{greensimprimitivitytheorem}), we can adapt all the proofs in Section \ref{identification} to the maximal crossed product:

\begin{thm}[cp. Theorem \ref{mainidentification}]
	The indicated maps in the following diagram are isomorphisms.
		
		\[\begin{tikzcd}
				&H^G_*(\mathcal E_{\Fin}G,\K_{\tilde A,\max}^G)\arrow[r,"\operatorname{pr}_*","\cong"']\arrow[d,"D_*","\cong"']	&H_*^G(\pt,\K_{\tilde A,\max}^G)\arrow[d,"D_*"]\\
				&H_*^G(\mathcal E_{\Fin}G,\K_{A,\max}^G)\arrow[r,"\operatorname{pr}_*"]&H^G_*(\pt,\K_{A,\max}^G)
		\end{tikzcd}\]
	In particular, the maximal Meyer--Nest assembly map can be identified with the maximal Davis--L\"uck assembly map. 
\end{thm}

\begin{cor}
	For any exotic crossed product functor $\rtimes_\mu G$, the $\mu$-Meyer--Nest assembly map can be identified with the $\mu$-Davis--L\"uck assembly map. 
\end{cor}

\subsection*{Acknowledgement}
This work arose from the author's master thesis at the University of M\"unster. The author would like to thank Siegfried Echterhoff, Michael Joachim, Thomas Nikolaus, Markus Schmetkamp and Felix Janssen for valuable discussions, comments, motivation and inspiration.

\end{document}